\documentclass{article}
\usepackage[utf8]{inputenc}
\usepackage{amssymb}
\usepackage{times}
\usepackage[T1]{fontenc}
\usepackage{enumitem}
\usepackage{amsthm}
\usepackage{amsmath}
\usepackage{graphicx}
\usepackage{varwidth}
\usepackage{natbib}

\usepackage{amssymb}
\usepackage[utf8]{inputenc}
\usepackage{amssymb}
\usepackage{times}
\usepackage[T1]{fontenc}
\usepackage{enumitem}
\usepackage{amsthm}
\usepackage{amsmath}
\usepackage{graphicx}
\usepackage{varwidth}

\newtheorem{theorem}{Theorem}
\theoremstyle{definition}
\newtheorem{ardef}{Definition}
\theoremstyle{plain}

\newtheorem*{exo}{Example}
\newtheorem{lemma}{Lemma}
\newtheorem{obs}{Observation}

\newtheorem{property}{Property}
\newtheorem{remark}{Remark}

\author{Krystian JOBCZYK and Mirna D\v{Z}AMONJA}
\title{The Lindstr\o m's Characterizability of Abstract Logic Systems \\for Analytic Structures Based on Measures}

\begin{document}
\maketitle




\begin{abstract}
In 1969, Per Lindstr\o m proved his celebrated theorem characterising the first-order logic and established criteria for the first-order definability of formal theories for discrete structures. K. J. Barwise, S. Shelah, J. V\"a\"an\"anen and others extended Lindstr\o m's characterizability program to classes of infinitary logic systems, including a recent paper by M. D\v{z}amonja and J. V\"a\"an\"anen on Karp's chain logic, which satisfies interpolation, undefinability of well-order, and is maximal in the class of logic systems with these properties. The novelty of the chain logic is in its new definition of satisfability. In our paper, we give a framework for Lindstr\o m's type characterizability of predicate logic systems interpreted semantically in models with objects based on measures (analytic structures). In particular, H\'{a}jek's Logic of Integral is redefined as an abstract logic with a new type of H\'{a}jek's satisfiability and constitutes a maximal logic in the class of logic systems for describing analytic structures with Lebesgue integrals and satisfying compactness, elementary chain condition, and weak negation.

    

\end{abstract}






\section{Introduction}
\label{sec:sample1}

In \cite{lindstrom1969}, Per Lindstr\o m elaborated criteria of the first-order characterizability of formal theories according to his previous research from \cite{lindstrom1966} on the so-called Lindsr\o m's quantifiers. Due to this theorem- a logical system shares the same expressive power with the first-order (elementary) logic when both the compactness theorem and downward Skolem-Loewenheim Theorem hold for it. 
The purely semantic proof of Lindstr\o m's theorem initiated research on the so-called abstract model theory - axiomatically initially depicted in \cite{barwise1977} and broadly described in  \cite{barwisefeferman} and in many other papers and monographs, such as \cite{ebbinghaus1985}.   

In essence, Lindstr\o m's results from the 60s strongly influenced two parallel research paths in the conceptual framework of abstract model theory. The first path -- renewed in the spirit of results from \cite{lindstrom1966} -- refers to infinitary logic -- very recently discussed in \cite{dzamonja} in some reference to Shelah's infinitary logic from \cite{shelah2012}  -- and generalized quantifiers initially introduced in \cite{mostowski} and broadly discussed, for instance, in \cite{barwisecooper}. The second research path refers to the original Lindstr\o m's theorem itself from \cite{lindstrom1969} and includes all the theorem's reconstruction attempts. 

Lindstr\o m's theorem has been extrapolated for various logical systems, for instance, for modal logic systems as in \cite{benthem2007,derijke}. Simultaneously, a significant part of the proof machinery has been incorporated into the model theoretical structure of functional analysis - due to ideas from \cite{henson,hensonheinrich}. In the same spirit, the Lindstr\o m's type characterization of analytic structures with a new approximation satisfaction relation was elaborated in \cite{iovino,xcaicedo}.

\subsection{The Paper Motivation and its Objectives}
 
Although the model-theoretic treatment of (even sophisticated) analytic structures, such as Banach spaces, seems to be well-grounded since the earlier works from the '70s, such as: \cite{henson,hensonheinrich,shelah},  Lindstrom's-type characterization itself of the analytic structures has been elaborated for a small number of types of the structures, such as continuous metric spaces \cite{benyacov,xcaicedo}. More precisely, the maximality logic and Lindstr\o m theorem were established here for the so-called Pavelka-Lukasiewicz logic - due to \cite{hajek1998book}. Unfortunately, no attempt at the first-order characterizability of analytic structures based on measures has yet been proposed. In particular, we have no information about the appropriate maximal logic suitable to describe analytic structures with integrals based on Lebesgue or - more generally - Radon's measure. Fortunately, a convenient formal bridgehead for such a construction has already been proposed in \cite{hajek1998book} in the form of the so-called \textit{Rational Pavelka Predicate Logic with Integrals} -- for simplicity -- to be called \textit{H\'{a}jek Integral Logic} in the paper. 

Due to this lack and shortcoming -- this paper is aimed at: 
\begin{description}
\item[$\bullet$] proposing a new abstract logic concept-based depiction of H\'{a}jek Integral Logic (HLI),
\item[$\bullet$] formulating and proving the Lindsr\o m's type theorem for the appropriate 'minimal' logic associated with the abstract logic --  previously defined for HLI.   
\end{description}
This paper forms an extended, and modified version of the conference paper \cite{jobczyk2021}.

\section{The terminological framework of the paper analysis}

Before we move to the proper part of the paper, we put forward a general conceptual framework for further analysis. At first - the formal definition of the abstract logic of a given signature and a couple of the close-related concepts will be recalled. 
Secondly, H\'{a}jek Integral Logic  -- as a unique extension of the so-called Predicate Pavelka-H\'{a}jek Logic -- will be described both syntactically and semantically.

\subsection{Abstract logic and its model-theoretic properties}

The formal definition of abstract logic - as a core notion of abstract model theory - was elaborated by P. Lindstr{\o}m in his famous work \cite{lindstrom1969}. The current depiction of this concept incorporates an approach from \cite{iovino}.

\begin{ardef} A\textit{logic} $\mathcal{L}$ is a triple $(\mathcal{K},\, Sent_{\mathcal{L}}, \models_{\mathcal{K}})$, 
where $\mathcal{K}$ is a class of structures of a given type\footnote{P. Lindstrom considered classical, i.e. discrete structures in \cite{lindstrom1966}. In \cite{xcaicedo}, continuous metric structures were considered.} closed under isomorphism, renaming and reducts, $Sent_{\mathcal{L}}$ is a function which assigns to each vocabulary $S$ a set $Sent_{\mathcal{L}}(S)$, i.e. a set of $\mathcal{S}$-sentences of $\mathcal{L}$ and $\models_{\mathcal{L}}$ is a satisfaction relation such that the following condition hold:
\begin{enumerate}
    \item $S\subseteq S^{'}$, then $Sent_{\mathcal{L}}(S)\subseteq Sent_{\mathcal{L}}(S^{'})$ (monotonicity of $Sent_{\mathcal{L}}$-operator),
    \item If $M\models_{\mathcal{L}}\phi$ (a formula $\phi$ and M remain in $\models$-relation), then there is a vocabulary $S$ such that M is an S-structure in $\mathcal{K}$ and $\phi$ is an S-sentence.  
    \item (Isomorphism property) If $M\sim N$ are isomorphic structures in $\mathcal{K}$ and $\phi$ is an S-sentence, then $\phi^{N} = \phi^{M}$.
    \item (Reduct property) If $S\subseteq S^{'}$, where both $S$ and $S^{'}$ are vocabularies, and $M$ is an $S^{'}$-structure in $\mathcal{K}$, then 
    \begin{equation*}
        M\models_{\mathcal{L}}\phi\,\,\,\mathrm{if\,\, and\,\, only\,\, if}\,\, (M\vert S)\models_{\mathcal{L}}\phi.
        \end{equation*}
    \item (Renaming property) If $\rho: S\to S^{'}$ is a renaming between vocabularies $S$ and $S^{'}$, then for each $S$-sentence $\phi$ there exists $S^{'}$-sentence $\phi^{\rho}$ such that
     \begin{equation*}
        M\models_{\mathcal{L}}\phi\,\,\mathrm{if\,\,and\,\,only\,\,if}\,\, M^{\rho}\models_{\mathcal{L}}\phi^{\rho},   \end{equation*}
    for each S-structure $M$ in the structure class $\mathcal{K}$\footnote{As usual, $M^{\rho}$ denotes $S^{'}$-structure obtained from $M$-structure by its converting through $\rho^{'}$.}. 
\end{enumerate}
\end{ardef}
As usual, if $M\models_{\mathcal{L}}\phi$, we say that $\phi$ is satisfied in $M$ and $M$ is said to be a \textit{model} for $\phi$. 
\begin{ardef}
An abstract logic $\mathcal{L}$ is said to be \texttt{closed under conjunction} if for all $\mathcal{L}$-sentences $\phi, \psi$ there is an $\mathcal{L}$-sentence $\phi\wedge \psi$ such that

\begin{equation*}
    M\models_{\mathcal{L}}\phi\wedge M\models_{\mathcal{L}}\psi\,\,\mathrm{if\,\,and \,\, only\,\, if}\,\, M\models_{\mathcal{L}} \psi\wedge \phi.
\end{equation*}
\end{ardef}
\begin{ardef}
An abstract logic is said to be \texttt{closed under negation} if for all $\mathcal{L}$-sentences $\phi$: 
\begin{equation}
    M\models_{\mathcal{L}}\neg\phi\,\mathrm{if\,\,and\,\ only\,\, if}\,\, M\not\models_{\mathcal{L}}\phi,
\end{equation}
for all M-structures in $\mathcal{K}$.
\end{ardef}
The general definition of abstract logic $\mathcal{L}$ allows us to define a (formal) \textit{theory} $T$ of $\mathcal{L}$.
\begin{ardef}
Let $\mathcal{L}$ be an abstract logic and let $S$ be a vocabulary. 
 An $S$-theory (or simply: a theory) $T$ is a set of all $S$-sentences of $\mathcal{L}$. 
\end{ardef}    
    \begin{ardef}
     Let $T$ be an S-theory of $\mathcal{L}$. If also $M$ be an S-structure in $\mathcal{K}$ such that $M\models_{\mathcal{L}}\phi$, for each $\phi\in T$, then $M$ is a \texttt{model} for $T$, what we denote by $M\models_{\mathcal{L}} T$. 
\end{ardef}
\begin{theorem}
If $T$ is consistent, then it has a model.
\end{theorem} 

\subsection{From {\L}ukasiewicz Logic to H\'{a}jek Integral Logic}

The conceptual tissue of abstract logic in its model-theoretic depiction was adopted in \cite{xcaicedo} to describe the Pavelka-H\'{a}jek Logic (in both the propositional and the predicate variant) on a base of \L ukasiewicz Logic. 
In the second part of this chapter, H\'{a}jek Logic of Integrals -- as an extension of the predicate Pavelka-H\'{a}jek logic -- is presented due to \cite{hajek1998book}. Whereas Pavelka-H\'{a}jek Logic in its two variants will be already described in terms of a conceptual tissue -- based on  the abstract logic concept -- due to \cite{iovino}, H\'{a}jek Logic of Integrals alone will be still presented classically -- due to \cite{hajek1998book}. The appropriate abstract logic concept-based apparatus for it will be elaborated in Sections 3 and 4.

\subsubsection{Pavelka Rational Logic}

Predicate Pavelka Rational Logic (PrePRL) constitutes a conservative extension of (predicate) Infinitely Valued \L ukasiewicz  Logic -- due to \cite{shepherdson} --- and extends this system by associating truth constants for rational in [0,1]. 

A similar relation holds between the propositional Infinitely Valued \L ukasiewicz  Logic and \textit{Pavelka Rational Logic} (PRL) as a propositional logic system, which forms the main subject of this paragraph. This similarity manifests itself in the way of defining PRL as an abstract logic. In fact, the appropriate class of structures for PRL is built up from the class of continuous metric structures being the appropriate class of structures for \L ukasiewicz Logic -- due to \cite{xcaicedo}. In addition, predicates of a PRL language take values from the closed interval $[0,1]$. 

Because of the fuzzy nature of PRL -- the usual $\models$-relation for an assignment relation $\mathcal{V}$ will be exchanged for an $\mathcal{V}$-assignment, which forms a counterpart of the so-called \textit{truth degree} defined in \cite{hajek1998book} and plays a role of a 'fuzzy satisfaction relation'. PRL as an abstract logic is introduced in two steps. At first, its 'surrogate' in the form of a \textit{Weak Pavelka Rational Logic} (WPRL) is put forward as a basis for the proper construction of PRL.


\begin{ardef} \texttt{A Weak Pavelka Rational Logic (WPRL)}\footnote{The name has not been used yet in the literature, although the current definition is introduced as defining [0,1]-\textit{valued logic} in def. 1.10. See: \cite{xcaicedo}, pp.1175-76.} is a triple

$(\mathcal{K},\, Sent_{\mathcal{WPRL}}, \mathcal{V})$, 
where $\mathcal{K}$ forms a class of metric continuous structures closed under isomorphism, renaming and reducts, $Sent_{\mathcal{L}}$ is a function which assigns to each vocabulary $S$ a set $Sent_{\mathcal{WPRL}}(S)$, i.e. a set of $\mathcal{S}$-sentences of $\mathcal{L}$  such that the following condition hold:
\begin{enumerate}
    \item $S\subseteq S^{'}$, then $Sent_{\mathcal{WPRL}}(S)\subseteq Sent_{\mathcal{WPRL}}(S^{'})$,
    \item A relation $\mathcal{V}$ assigns to each pair $(\phi, M)$, where $\phi$ is S-sentence and $M$ is S-structure in $\mathcal{K}$, a real number $\phi^{M}\in [0,1]$.   
    \item (Isomorphism property for WPRL) If $M\sim N$ and they are metrically isomorphic and $\phi$ is an S-sentence of $\mathcal{WPRL}$, then $\phi^{M} = \phi^{N}$.
    \item (Reduct property for WPRL) If $S\subseteq S^{'}$, where both $S$ and $S^{'}$ are vocabularies, and $M$ is an $S^{'}$-structure in $\mathcal{K}$ and $\phi$ is an S-sentence, then 
    $\phi^{M} = \phi^{M\restriction S}$ ($M\restriction S$ denotes a reduct $M$ to $S$).
    \item (Renaming property for WPRL) If $\rho: S\to S^{'}$ is a renaming between vocabularies $S$ and $S^{'}$, then for each $S$-sentence $\phi$ there exists $S^{'}$-sentence $\phi^{\rho}$ such that
     $\phi^{M} = (\phi^{\rho})^{M^{\rho}}$
    for each S-structure $M$ in the structure class $\mathcal{K}$\footnote{As usual, $M^{\rho}$ denotes $S^{'}$-structure obtained from $M$-structure by its converting through $\rho^{'}$.}. 
\end{enumerate}
\end{ardef}
\begin{ardef}
If M is S-structure from $\mathcal{K}$, and $\phi\in Sent_{\mathcal{WPRL}}$,  $\phi^{M} = 1$, then we say that M \texttt{satisfies} $\phi$ and write $M\models_{\mathcal{WPRL}} \phi$.  
\end{ardef}

Instead of $\mathcal{V}$, we will sometimes write $e$ to express the same assignment, in particular  - in the context when the true value of PRL-formulae is considered\footnote{It is not easy to see that this definition of abstract Pavelka Rational Logic is structurally similar to the Kripke frame-based depiction of semantics for many fuzzy logic systems, such as \textbf{S5}[0,1]- due to \cite{hajek1998book}. In fact, their semantics is usually given by the triples: $(W, R, e)$, where $W$ is a non-empty set of states, $R$ is an accessibility relation $\subseteq W\times W$ and $e: \mathcal{L}\to [0,1]$ is an evaluation function, for a given language of a given fuzzy logic system.}.

For a complementary definition of PRL on a base of WPRL, one needs to close the abstract logic $(\mathcal{K},\, Sent_{\mathcal{WPRL}},\, \mathcal{V})$ under some connectives. The key connective for each [0,1]-valued logic -- based on \L ukasiewicz logic -- is the so-called \textit{\L ukasiewicz implication}. It forms a  function from $[0,1]^{2}$ into $[0,1]$ defined by the clause: 
\begin{equation}
    x\to_{L} y = \min\{1 - x+ y, 1\}.
\end{equation}
It is noteworthy to state that $x\to_{L}y = 1$ if and only of $x\leq y$.

  \begin{ardef} We say a [0,1]-valued logic $\mathcal{L}$ is
\texttt{closed under the basic} \texttt{connectives} if and only if the following condition holds, for each vocabulary $S$:
\begin{enumerate}
    \item If $\phi, \psi\in Sent_{\mathcal{L}}(S)$, then there exists a sentence $\phi\to_{L}\psi\in Sent_{\mathcal{L}}(S)$ such that $(\phi\to_{L} \psi)^{M}$ if and only if $\phi^{M}\to_{L}\psi^{M}$, for every $S$-structure $M$.  
    \item For each element $r \in [0,1]$, the set $Sent_{\mathcal{L}}(S)$ contains a sentence with constant truth value $r$ to be called the constants of $\mathcal{L}$ and denoted by $\bar{r}$, i.e. $e(\bar{r}) = r$. 
\end{enumerate}
\end{ardef}
Having defined closeness under basic connectives we are in a position to define Pavelka Rational Logic as an abstract logic as follows.

\begin{ardef}
A \texttt{Pavelka Rational Logic}(PRL) is a triple $(\mathcal{K},\, Sent_{\mathcal{WPRL}},\, \mathcal{V})$ defined as in Def. 6 and closed under basic connectives. 
\end{ardef}

From the perspective of [0,1]-evaluation of formulae, the following observation seems to be noteworthy.  

\begin{remark}
Let M be an S-structure for abstract logic $\mathcal{L}=\mathcal{PRL}$. Then for any $\phi\in Sent_{\mathcal{PRL}}(S)$ hold: 
\begin{enumerate}
    \item $M\models_{\mathcal{L}} \phi\to_{L} \bar{r}\iff \phi^{M}\to r$ (alternatively: $e(\bar{r}\to_{L} \phi)= 1\iff e(\phi)\geq r$). 
    \item $M\models_{\mathcal{L}} \bar{r}\to_{L} \phi\iff r\to_{L} \phi^{M}$ (alternatively: $e(\phi\to_{L} \bar{r})= 1\iff e(\bar{r})\geq \phi$). 
\end{enumerate}
\end{remark} 
Note that for every $S$-structure $M$ it holds\footnote{These properties correspond to the axiom-based depiction of PRL - due to \cite{hajek1998book}:
\begin{enumerate}
    \item $(\bar{r}\to \bar{s} \equiv \bar{r\Rightarrow s})$
    \item $\neg\bar{r} \equiv \bar{1- r}$
\end{enumerate}}:
\begin{eqnarray*}
     (\neg\phi)^{M} = 1 - \phi^{M},\,\,
     (\phi\vee \psi)^{M} = \min\{\phi^{M}, \psi^{M}\},\,\,
     (\phi\wedge \psi)^{M} = \max\{\phi^{M}, \psi^{M}\}.
\end{eqnarray*}

It is noteworthy to underline that the \textit{truth value} of each 
$\phi\in Sent_{\mathcal{PRL}}$ is determined either by the set
\begin{equation}
    \{r\in \mathcal{Q}\cap [0,1]: M\models_{\mathcal{PRL}} \phi\to_{L} \bar{r}\} = 
\{r\in \mathcal{Q}\cap [0,1]: e(\bar{r}\to_{L} \phi) = 1\}
\end{equation}
or by the set 
\begin{equation}
  \{r\in \mathcal{Q}\cap [0,1]: M\models_{\mathcal{PRL}} \bar{r}\to_{L} \phi\}\footnote{Obviously, one can reformulate the definitions of the sets by using $\leq$-relation.}.  
\end{equation}
It relies on the observation that truth value of the formula (by fuzzy evaluation) is not too false, i.e. its truth value is \textit{at least} equal to $r$ (if $\bar{r}\to_{L} \phi$). Similarly, if $\phi\to_{L} \bar{r}$, then the truth value of $\phi$ is not too true, i.e. its truth value is \textit{at most} equal to $r$. 

\subsubsection{Predicate Pavelka Rational Logic (PrePRL)}

\textit{Predicate Pavelka Rational Logic} (PrePRL) forms a predicate enhancement of Pavelka Rational Logic, obtained from PRL if exchanged its propositional language is the predicate one. In this paragraph, we intend to provide an abstract logic-based depiction of PrePRL constructing the appropriate class of structures and a new set of sentences -- due to the taxonomy of terms and formulae in each predicate language. The construction itself will be proceeded by explaining the direction of the construction and the nature of the constructed structures.\\ 

\noindent\textbf{I. A couple of explanatory remarks.} Before we move to details of the construction, some explanatory remarks concerning the direction of this construction and the nature of the structures themselves should be previously made.
\begin{enumerate}
    \item While a majority of two-valued predicate logics include the predicate "=" among their formulae -- interpreted semantically by identity relation, a multi-valued predicate logic system should include a sign "$\approx$" intentionally interpreted in semantics by  \textit{similarity relation} as fuzzy identity. PrePRL forms an example of such a system. 
    \item  Meanwhile, we generally have two possibilities to formally grasp the '='-relation symbol in a predicate language. The first way relies on listing its formal properties: reflexivity, symmetry and transitivity -- as this way stems from the observation that it forms just an equivalence relation. The alternative, stronger way relies on defining the appropriate congruences\footnote{Note that each congruence is a unique equivalence relation. It justifies why this method may be described as the stronger one.} for formulae of the predicate language. 
\end{enumerate}
More precisely, for variable symbols $x_{1}, \ldots x_{n}$ and $y_{1},\ldots, y_{n}$, the '='-relation symbol is described by the following conjunction of congruences for each $n$-ary operation symbol $f$:

\begin{equation}
   (x_{1}= y_{1}\wedge\ldots x_{n}= y_{n}) \to (f(x_{1},\ldots, x_{n})= f(y_{1},\ldots, y_{n}).  
\end{equation}
and for $n$-ary predicate $P$:
\begin{equation}
  (x_{1}= y_{1}\wedge\ldots x_{n}= y_{n}) \to (P(x_{1},\ldots, x_{n})\to_{L} P(y_{1},\ldots, y_{n})).    
\end{equation}
The same concept of congruence -- as a fuzzy counterpart of identity -- may be adopted to grasp $\approx$-relation symbol -- semantically interpreted by similarity relation as follows: 
\begin{equation}
   (x_{1}\approx y_{1}\wedge\ldots x_{n}\approx y_{n}) \to_{L} (f(x_{1},\ldots, x_{n})\approx f(y_{1},\ldots, y_{n}).  
\end{equation}
and for $n$-ary predicate $P$:
\begin{equation}
  (x_{1}\approx y_{1}\wedge\ldots x_{n}\approx y_{n}) \to_{L} (P(x_{1},\ldots, x_{n})\approx P(y_{1},\ldots, y_{n})).    
\end{equation}

\noindent\textbf{II. PrePRL as an abstract logic.} In order to provide an abstract logic concept-based depiction of PrePRL, let us repeat that PRL is semantically represented by continuous metric spaces with the 1-Lipschitz condition (see: \cite{xcaicedo}, pp.1173, 1179.)\footnote{The authors of the paper say about \L ukasiewicz-Pavelka Logic.}. 
\begin{ardef}{\textbf{The 1-Lipschitz condition}.}
Let $S$ be a class of continuous metric structures, and let $M$ be such an $S$-structure. For all $n$-ary predicates $P$, all $n$-ary function symbols $f$, $\bar{a} = a_{1},\ldots a_{n}\in M$, $\bar{b} = b_{1},\ldots b_{n}\in M$
\begin{equation}
    d(f(\bar{a}),f(\bar{b})), \leq \sup_{1\leq i\leq n}d(a_{i}, b_{i}),
\end{equation}
\begin{equation}
    d(P(\bar{a}),P(\bar{a}))\leq \sup_{1\leq i\leq n}d(a_{i}, b_{i}).
\end{equation}
\end{ardef}
Meanwhile, the similarity relation $x\approx y$ corresponds to the distance $d(x,y)$ such that $d(x,y) = 1- x\approx y$, i.e. both relations are their mutual negations\footnote{It is explainable in the light of the observation that a large distance between two elements $x$ and $y$ means a small similarity between them and conversely.}. It infers a unique similarity-based Lipschitz condition in the following form.
\begin{ardef}{\textbf{The Similarity-based Lipschitz Condition}\footnote{In fact, we could name this property as an anti-contraction}.} Let $f$ be a function symbol, S - be a class of structures with similarity relation $\approx$, $M$ be an S-structure and $\bar{a} = a_{1},\ldots, a_{n}\in M$ and $\bar{b} = b_{1},\ldots, b_{n}\in M$. Then it holds:
\begin{equation}
  (\bar{a}\approx \bar{b})\leq \inf_{1\leq i\leq n}\{ f(a_{i})\approx f(b_{i})\}\footnote{Directly from the relation $d(x,y) = 1 - x\approx y$, one can infer that $f(\bar{a})\approx f(\bar{b}) \leq sup_{1\leq i\leq n}\{a_{i}\approx b_{i}\}$. After some modifications, it also allows us to infer the above condition.}.
\end{equation}
\end{ardef}
As a result, the adequate $S$-class for PrePRL is a class of continuous, similarity-based Lipschitz structures. 

Having already established the class of the adequate structures for PrePRL, we are in a position to introduce a set of PrePRL-sentences $Sent_{\mathcal{PrePRL}}$ -- defined similarly like a set of formulae of a predicate language. More precisely, we will identify $Sent_{\mathcal{PrePRL}}$ with a set of formulae of $\mathcal{L}(PrePRL)$, which is built up from terms of $\mathcal{L}(PrePRL)$, as usual. 
\begin{ardef}{(\textbf{Terms of} $\mathcal{L}(PrePRL)$}. Let us assume that a class of structure (of a type) $S$ is given and let I be a non-empty set of indices. 
The set of $S$-terms $Tm_{\mathcal{L}(PrePRL)}$ of $\mathcal{L}(PrePRL)$ is defined inductively as follows:
\begin{equation}
 Tm_{0} = X \cup C,  
\end{equation}
where $X$ is a finite set of variables and $C$ is a set of constants, 
\begin{equation}
Tm_{l+1} = Tm_{l} \cup \{\langle f_{j}(t_{1},\ldots, t_{m}) \rangle\},  
\end{equation}
where $f_{j}$ is an $j$th function  symbol, for $j\in I$, and 
\begin{equation}
    Tm = \bigcup_{l\in\mathbb{N}}Tm_{l}.
\end{equation}
\end{ardef}
\begin{ardef}{(\textbf{Formulae of $\mathcal{L}(PrePRL)$})}.
Let us assume that a class of structures (of a type $S$) with an $S$- model $M$ is given, let $I$ be a (non-empty) set of indices and let $P_{i}$ be $i$th  $n$-ary predicate and $\widehat{r}$ forms a constant symbol for a rational $r$.     
The set $Form_{\mathcal{L}(PrePRL)}$ is defined inductively as follows:\\
\begin{eqnarray*}
 Fm_{0} =\{\langle t, \approx, s \rangle: t,s\in Tm_{\mathcal{L}(PrePRL)}\}\cup\\
 \cup\{P_{i}(t_{1},\ldots, t_{n}): i\in I\,\, \mathrm{and}\,\, t_{1},\ldots t_{n}\in Tm\}\\ \cup\{\widehat{r}: \widehat{r}^{M} = r\in \mathbb{Q}\cap [0,1]\}, 
\end{eqnarray*}
\begin{eqnarray*}
 Fm_{l+1} = Fm_{l}\cup \{\neg F: F\in Fm_{l}\}\cup\\
 \cup\{F\to_{L} G: F,G\in Fm_{l}\}\cup\\
 \cup\{\forall_{x_{n}}F: F\in Fm_{l}\,\, \mathrm{and}\,\, n\in\mathbb{N}\}.   
\end{eqnarray*}
\end{ardef}

Having established a class of formulae $\mathcal{L}(PrePRL)$, we are in a position to define PrePRL as an abstract logic. 
\begin{ardef}{(\textbf{Predicate Pavelka Rational Logic (PrePRL)}.)} Let $S$ be a predicate vocabulary.
\texttt{A Predicate Pavelka Rational Logic} is a triple $(\mathcal{K}, Sent_{PrePRL}, \mathcal{V})$, where:
\begin{description}
\item[a)]$\mathcal{K}$ is a class of continuous, similarity-based Lipschitz structures closed under renaming, reducts and isomorphism, 
\item[b)] $Sent_{PrePRL}(S)$ is a class of formulae of $\mathcal{L}$(PrePRL) -- as defined in Def. 13,
\item[c)] a relation $\mathcal{V}$ assigned to each pair $(\phi, M)$, where  $\phi$ is $S$-sentence and $M$ is $S$-structure in $\mathcal{K}$, a real number $\phi^{M} \in [0,1]$, 
\end{description} 
which is closed on basic connectives and existential (general) quantifiers.
\end{ardef}

\section{Hajek Logic of Integrals}
In this section, HLI is introduced both syntactically and semantically -- due to its axiomatic depiction from \cite{hajek1998book} -- as an extension of Predicate Pavelka Rational Logic (see:\cite{hajek1998book,xcaicedo}). This depiction will constitute a convenient bridgehead to reformulate this system as an abstract logic.  \\  

\noindent\textbf{I. Syntax of HLI}. $HLI$ is defined both syntactically and semantically in \cite{hajek1998book} in a language $\mathcal{L}(HLI)$ given by the grammar:
\begin{eqnarray*}
 \phi:= \phi\vert \neg \phi\,\, \vert\,\, \phi\vee \psi\,\,\vert\,\phi\veebar\psi\,\vert\, \phi\wedge\psi  \,\vert\, \phi \& \psi \,\vert\forall_{x} \phi\,\,\vert\,\, \int \phi \mathrm{d}x\,\vert\footnote{$\veebar$ denotes the so-called strong disjunction, but $\&$ denotes strong conjunction. Their semantic meaning will be explained in the context of semantic models for HLI. }\,\\\forall_{x}\phi=\forall_{y}\psi\,\vert\, \displaystyle\int\phi\mathrm{d}x=\displaystyle\int\psi\mathrm{d}y.   
\end{eqnarray*}

\noindent In other words, $\mathcal{L}(HLI)$ extends a significant part of PrePRL-language (without $\approx$-symbol exchanged for '=')\footnote{A sense of this modification follows from another approach to the representation of fuzziness, which constitutes a founding conceptual idea of HLI. Fuzziness is rendered less explicit and manifests itself, for example, by truth values of formulae of this system.} by a new quantifier $\displaystyle\int$ (read 'probably') and extends the definition of PrePRL-formulae\footnote{In fact, defining PrePRL in the previous chapter, and we said about \L ukasiewicz implication $\to_{L}$, which is definable either in terms of $\neg$ and strong disjunction $\veebar$ or - in terms of $\neg$ and strong conjunction $\&$.} by the clause informing that if $\phi$ is a formula and $x$ is a variable, then $\displaystyle\int \phi\mathrm{d}x$. 
 $HLI$ is \textit{syntactically} determined in $\mathcal{L}(HLI)$ by:
\begin{description} 
\item[$\bullet$] Axioms:
\begin{description}
    \item[($\mu 1$)] $\displaystyle\int v\mathrm{d}x \equiv v$\,\, for $v$ not not containing $x$ freely,
\item[($\mu 2$)] $\displaystyle\int (\neg \phi)\mathrm{d}x \equiv \neg \displaystyle\int \phi\mathrm{d}x$, 
\item [($\mu 3$)] $\displaystyle\int (\phi\to \psi)\mathrm{d}x \equiv  \displaystyle\int \phi\mathrm{d}x\to\displaystyle\int \psi\mathrm{d}x$,
\item[($\mu 4$)] $\displaystyle\int (\phi\veebar\psi)\mathrm{d}x \equiv \Big((\displaystyle\int\phi\mathrm{d}x\to \displaystyle\int (\phi \& \psi)\mathrm{d}x)\to \displaystyle\int\psi\mathrm{d}x \Big)$,
\item[($\mu$5)] $\displaystyle\int(\displaystyle\int \phi\mathrm{d}x)\mathrm{d}y \equiv \displaystyle\int(\displaystyle\int \phi\mathrm{d}y)\mathrm{d}x$.\\
\end{description}
\item[$\bullet$] Inference rules: \textit{Modus Ponens}, substitution, generalization and
$$\frac{\phi}{\displaystyle\int \phi\mathrm{d}x}\,\,\,, \frac{\phi\to \psi}{\displaystyle\int \phi\mathrm{d}x\to \displaystyle\int\psi\mathrm{d}x}.$$
\end{description}
\begin{ardef} HLI -- in a language, $\mathcal{L}(HLI)$ -- is defined as the smallest logical systems consisting of axioms $(\mu 1)$-$(\mu 5)$ and closed under the above 'integral' inference rules and the inference rules of PrePRL (MP and generalization). 
\end{ardef}
\begin{exo}
\begin{enumerate}
    \item $\displaystyle\int(\displaystyle\int\neg\phi\veebar\psi\,\mathrm{d}x)\mathrm{dy}$ is well-founded $\mathcal{L}(HLI)$-formula, but $\displaystyle\int \phi\mathrm{d}x + \displaystyle\int\psi\mathrm{dx}$ does not\footnote{However, the additivity properties of integrals are rendered by the axioms. For example, finite additivity of integrals is rendered by $(\mu 4)$. It may be seen in the perspective of semantic interpretations of logical connectives by the appropriate $\min$ and $\max$ norms.}.
    \item $\forall _{x}\phi = \forall_{y}\psi$ is a legal $\mathcal{L}(HLI)$-formula, but $\phi= \psi$ does not. This formula is prohibited as '=' may be used only for quantifiers.  
\end{enumerate}
\end{exo}
\noindent\textbf{II. Semantics} -\cite{hajek1998book}, p. 238-240.
The semantics of HLI - described in \cite{hajek1998book}, pp. 238-240 -- was elaborated in terms of the so-called \textit{probabilistic models} of a general form:
\begin{equation}
    M = \langle\vert M\vert, R_{1}^{M},\ldots R_{k}^{M}, f_{1}^{M},\ldots, f_{n}^{M}, c_{1},\ldots, c_{l}\,, \mu \rangle, 
\end{equation}
where $\vert M\vert$ is a \textit{countable} or \textit{finite} set (the model universe), $R_{i}^{M}, f_{j}^{M}$, for $i\in 1, \ldots, k$, $j\in 1,\ldots,l$, interpret predicates and function symbols (\textit{resp.}) as usual (over the real unit $[0,1]$) and $c_{1}^{M},\ldots, c_{m}^{M}$ interpret constant symbols in the same way over the real unit $[0,1]$ and $\mu$ is a probability measure on $M$\footnote{That is $\mu$ is a function assigning to $m\in \vert M\vert$ a real $\mu(m)\in [0,1]$ such that $\sum_{m\in\vert M\vert} \mu(m) = 1$. For an arbitrary subset $A\subseteq \vert M\vert: \mu(A) = \sum_{m\in A} \mu(m).$ More generally, if $A\subseteq \vert M\vert^{n}$ then $\mu(A) = \sum\{\mu(m_{1})\mu(m_{2})\ldots\mu(m_{n}): \langle m_{1},\ldots, m_{n}\rangle\in A\}$.}

In fact, the proper model suitable to interpret HLI must form a refinement of the weak probabilistic model given by (18). Instead of $\mu$, one needs to consider the semantic counterpart of the integral formula. Let us repeat that the semantic integral\footnote{In further part, we will use simply the name 'integral' if it does not generate any confusion.} is defined as follows.

For each function $f: \vert M\vert \to [0,1]$ the (Lebesgue) integral $\displaystyle\int f \mathrm{d}\mu$ is defined by $\sum _{n\in\vert M\vert} f(m)\mu(m)$. More formally, if a measure space $(X, \mathcal{A}, \mu)$ is given, where $\mathcal{A}$ forms an $\rho$-algebra of subsets of $X$, then each finite, pairwise disjoint family $\{A_{1}, A_{2},\ldots, A_{n}\}\subseteq \mathcal{A}$ such that $\bigcup_{k=1}^{n} A_{k} = X$ is said to be a \textit{measurable dissection of $X$.} It allows us to define the \textit{Lebesque integral} of (a given) function $f$.
\begin{ardef}{\textbf{The Lebesgue integral}[Hewitt],p.164.} Let $f$ be any function from $X$ to $[0,\infty)$. Then the \textit{Lebesgue integral} $L(f)$ of $f$ is defined as follows.
\begin{eqnarray*}
    L(f) = \sup\Big\{\sum_{k=1}^{n} \inf\{f(x): x\in A_{k}\}\mu(A_{k}): \{A_{1},\ldots, A_{n}\}\\
    \mathrm{is\,\, a\,\, measurable\,\, dissection\,\, of\,\, X} \Big\}.
\end{eqnarray*}
 
\end{ardef}

It is noteworthy that if $f$ is two-valued, i.e. $range (f)\subseteq \{0,1\}$, then $\displaystyle\int f\mathrm{d}\mu = \mu(f^{-1}(1)) = \mu(\{m: f(m)=1\})$ (it is a measure of the set whose characteristic function is $f$).

These observations will be exploited now for defining the 'semantic integrals'\footnote{The authors of \cite{hajek1998book} are said to be called the 'weak integrals' because of being components of the so-called 'weak models'.} in the algebraic environment of $\mathcal{F}$-algebra of [0,1]-fuzzy subsets of weak probabilistic model. Its precise definition is as follows. 
\begin{ardef}{(\textbf{$\mathcal{F}$-algebra} of weak probabilistic models.)} Let M be a weak probabilistic model with a universe $\vert M\vert$. $\mathcal{F}$-algebra of (weak probabilistic) model M is a family of [0,1]-fuzzy subsets\footnote{Note that these sets are determined by their characteristic functions, so we can think about them as about their functions, what justifies this defining method.} of $\vert M\vert$ containing each constant rational function with the value $r\in [0,1]$ and closed under $\Rightarrow$ (if $f, g\in\mathcal{F}$ and $h(a) = f(a)\Rightarrow g(a)$, for all $a\in \vert M\vert$ then $h\in\mathcal{F}$.)\footnote{One can infer from this that $\mathcal{F}$ is closed under $\neg \wedge, \vee, \veebar$.(see: \cite{hajek1998book}, p.240.) It corresponds well with the closeness of the set of HLI formulae under their corresponding connectives.} 
\end{ardef}
\begin{ardef}{(\textbf{Semantic Integrals}).\cite{hajek1998book}, p. 240.} Let us assume that a weak probabilistic model M with a non-empty domain $\vert M\vert$ with its $\mathcal{F}$-algebra is given. A \textit{semantic integral} on $\mathcal{F}$is a mapping $\oint$ associating to each $f\in \mathcal{F}$ its Lebesgue integral -- constructed as described in Def. 16 -- $\displaystyle\oint f\mathrm{d}x\in [0,1]$ and satisfying the following conditions(if $f, g\in\mathcal{F}$, $k_{r}$ is the constant function with the value $r$, $h: \vert M\vert\times \vert M\vert \to [0,1]$ is such that for each $a, b\in \vert M\vert$, the functions $f_{b}(x) = h(x,b)$ and $g_{a}(y) = h(a, y)$ are both in $\mathcal{F}$):    
\begin{equation}
    \displaystyle\oint k_{r}\mathrm{d}x = r,
\end{equation}
\begin{equation}
 \displaystyle\oint (1- f)\mathrm{d}x = 1 - \displaystyle\oint f\mathrm{d}x,   
\end{equation}
\begin{equation}
 \displaystyle\oint(f\Rightarrow g(\mathrm{d}x \leq (\displaystyle\oint f\mathrm{d}x\Rightarrow \displaystyle\oint g\mathrm{d}x)   
\end{equation}
\begin{equation}
    \displaystyle\oint(f\oplus g)\mathrm{d}x = \displaystyle\oint f\mathrm{d}x + \displaystyle\oint g\mathrm{d}x - \displaystyle\oint (f \ast g)\mathrm{d}x,
\end{equation}
\begin{equation}
 \displaystyle\oint(\displaystyle\oint h\mathrm{d}x)\mathrm{d}y = \displaystyle\oint (\displaystyle\oint h\mathrm{d}y)\mathrm{d}x,\,\, \mathrm{if\,\, both\,\, sides\,\, defined.}   
\end{equation}
\end{ardef}
Having defined 'semantic integrals,' we are in a position to define the so-called \textit{weak probabilistic model} as a structure slightly modifying the probabilistic model.
\begin{ardef}{\textbf{Weak probabilistic model}.} Each structure of the type:
\begin{equation}
 M = \langle\vert M\vert, R_{1}^{M},\ldots R_{k}^{M}, f_{1}^{M},\ldots, f_{n}^{M}, c_{1},\ldots, c_{l}\,, \displaystyle\oint\rangle\,, 
\end{equation}
where all components are defined as previously and $\displaystyle\oint$ is a semantic integral defined on $\mathcal{F}$-algebra of fuzzy subsets of M. 
\end{ardef}
In order to build a correspondence between HLI-axioms and axioms $(\mu 1)$-$(\mu 5)$ for semantic integrals, we need to complement the standard truth value function $\Vert \bullet\Vert_{M}$  for $\mathcal{L}(HLI)$ expressions with fuzzy connectives by the new  clause for the integral-type formulae:
\begin{equation}
    \Vert \displaystyle\int\phi \mathrm{d}x\Vert_{M} = \displaystyle\oint\Vert \phi\Vert_{M, v}\mathrm{d}x,
\end{equation}
provided that $\Vert \phi\Vert_{M}\in \mathcal{F}$. (otherwise -- undefined).

\section{Towards H\'{a}jek Logic of Integrals as an Abstract Logic}

In order to describe HLI as an abstract logic -- due to \cite{lindstrom1966,xcaicedo} -- for the use of its Lindstr\o m's characterization, one needs to introduce a satisfaction relation for $\mathcal{L}(HLI)$-formulae in weak probabilistic models. In this chapter, a new satisfaction relation to be called \textit{H\'{a}jek's satisfiability} (symbolically: $\models^{H}$) for $\mathcal{L}(HLI)$ is introduced\footnote{The analysis of the chapter are inspired by Iovino's ideas from \cite{iovino}. However, Iovino's ideas refer to the standard definition of satisfaction relation and are not immersed in the analytic conceptual scenario as the current considerations.}.  

In a standard way, a fuzzy logic formula $\phi$ is satisfied in a given model $\mathcal{M}$, in its state $w$ if and only if 1 (true value) is associated with $\phi$ in $w$ by an evaluation function $e$. Formally,
\begin{equation}
    \mathcal{M},w\models\phi\iff e(\phi, w) = 1.
\end{equation}
We preserve this standard and restrictive convention for all $\mathcal{L}(HLI)$-expressions except for quantifier-based expressions with '=,' when we will relax our expectations concerning satisfaction conditions in weak probabilistic models because of the nature of the integral quantifier expressions. It finds its reflection in H\'{a}jek's satisfiability for a majority of $\mathcal{L}(HLI)$-formulae given by the clauses, which develop (10).

We proceed to introduce its definition by short reasoning because of $\&$ connective. Since $e(\phi\&\psi) = 1\iff\max(0, e(\psi, w)+ e(\phi, w) - 1) =1$ because of $\Vert \phi\&\psi\Vert_{M} = \max\{0, x+y - 1\}$, we can infer that it must hold: $e(\phi, w) + e(\psi, w) - 1 = 1\iff e(\phi, w) = 1$ and $e(\psi, w) = 1$\footnote{It follows from the fact that both $e(\psi,w), e(\phi, w) \in[0,1]$. This result may be alternatively inferred from the fact that $\phi\&\psi$ stands for $\neg(\phi\to\psi)$, thus $e(\phi\&\psi w)=1\iff e(\neg(\phi\to\psi),w) = 1$, i.e. $e(\phi,w)= 1, e(\psi,w) = 1$.}. It leads to the following definition.

\begin{ardef}(H\'{a}jek's satisfaction for $\mathcal{L}(HLI)$-formulae without '='.) Let $\mathcal{M}$ be a weak probabilistic model and $w\in\vert\mathcal{M}\vert$ with a set of constants $C$.  
\begin{enumerate}
    \item $(\mathcal{M},w)\models^{H} \phi\iff e(\phi, w) = 1$.
    \item $(\mathcal{M},w)\models^{H} \neg\phi\iff e(\phi, w) = 0$.
    \item $(\mathcal{M},w)\models^{H} \phi\wedge\psi\iff \min\{e(\psi,w),e(\phi,w)\}) = 1$\footnote{Obviously, this condition may be reformulated to the condition that $e(\phi, w) =1, e(\psi, w)$. However, such a formulation hides the provenance of the condition. It is noteworthy that the satisfaction relation itself is not sensitive to a semantic difference between fuzzy connectives}.
    \item $(\mathcal{M},w)\models^{H}\phi\vee\phi\iff\max\{e(\psi,w),e(\phi,w)\}) = 1.$
    \item $(\mathcal{M},w)\models^{H}\phi\&\phi\iff e(\phi, w) = 1, e(\psi, w) = 1$.
    \item $(\mathcal{M},w)\models^{H}\forall_{x}\phi(x)\iff \inf_{c\in C}\{e(\phi(c),w)\}=1.$ 
    \item $(\mathcal{M},w)\models^{H}\exists_{x}\phi(x)\iff \sup_{c\in C}\{e(\phi(c),w)\} = 1.$ 
    \item $(\mathcal{M},w)\models^{H}\displaystyle\int \phi\mathrm{d}x\iff
    e(\displaystyle\int \phi\mathrm{d}x,w) = 1$\footnote{This situation should not be identified with a fact when a given Lebesgue integral is computed in some limits, and it takes a value 1. This value plays here a role of the highest \textit{truth value} for its corresponding 'integral expression in the sense of this concept given by \cite{hajek1998book}, pp. 238-241.}. 
\end{enumerate}
\end{ardef}
\begin{exo}
\begin{enumerate}
    \item Let us assume that $e(\phi, w) = 0.9$ and $e(\psi,w) =1$ in some w. p. model. Obviously, $\mathcal{M},w\models^{H} \phi\vee \psi$, but neither $\mathcal{M},w\models^{H} \phi\&\psi$ nor $\mathcal{M},w\models^{H} \phi\wedge\psi$ because $e(\phi, w) = 0.9$ violates H\'{a}jek's satisfaction in these cases due to 3) and 4). 
    \item Let us consider a model $\mathcal{M}$ with $\mathcal{C}$ with $C=\{\bar{1},\bar{2}, \bar{3}\}$, such that
$e(\phi(\bar{1})) = 0.5$, $e(\phi(\bar{2})) = 0,8$ and $e(\phi(\bar{3}))=1$. Than $\mathcal{M}\models^{H}\exists_{x}\phi(x)$ (as $\sup_{c\in C}\{e(\phi(c))\} = 1)$, but $\mathcal{M}\not\models^{H}\forall_{x}\phi(x)$ (as $\inf_{c\in C}\{e(\phi(c))\} = 0.5\not= 1)$.  
\end{enumerate}
\end{exo}
It remains to complement the definition of H\'{a}jek's satisfiability for '=' with $\mathcal{L}(HLI)$ quantifiers, in particular -- with the integral quantifiers.

\subsection{H\'{a}jek's satisfiability}

For the use of H\'{a}jek's satisfiability of '=' with $\mathcal{L}(HLI)$ quantifiers, we intend to refer to the concept of measure because of a measure-based nature o Lebesgue integrals. Informally speaking, we want to accept $\mathcal{L}(HLI)$-expressions of the type $\displaystyle\int f\mathrm{d}\mu = \displaystyle\int g\mathrm{d}\mu$ as H\'{a}jek's satisfiable in w.p. models if the $\mu$-measures of the appropriate $\mathcal{M}$-truth value sets for the left-side and the right-side integral expressions are almost identical. The definition of H\'{a}jek's satisfiability will be inductive -- due to the increasing complication degree of the measure. Initially, we begin with a Dirac's measure case to generalize the situation later.

\textbf{\textit{I. H\'{a}jek's satisfiability for '=' in $\{0,1\}$-Dirac's case.}}

For the use of H\'{a}jek's satisfiability, we will refer to the slightly modified definition of Lebesgue integral\footnote{In fact, it forms an approximation of the proper definition of integral given by Def. 16. (see: \cite{hewitt}, p. 164).}:  
$$D_{f}:\oint f\mathrm{d}\mu = \sum_{m\in M}f(m)\mu(\{m:f(m) =1\}).$$

If range($f$)$\in\{0,1\}$, then we obtain
$$ \oint f\mathrm{d}\mu = \mu\{x: f(x)=1\}\cdot 1 = \mu\{f^{-1}(1)\}.$$

Our intention is to accept that the statement of the form $\displaystyle\int f\mathrm{d}x = \displaystyle\int g\mathrm{d}x$ is \textit{H\'{a}jek's satisfiable} (in a $\{0,1\}$-Dirac's case) in a given weak probabilistic\footnote{This type of models plays a role of the intended models in the whole reasoning -- even if it is not clearly mentioned.} model $\mathcal{M}$ by a valuation $v$ if the measure of a difference set $\{x: f(x)=1\} - \{y: g(y) =1\}$ is arbitrarily small and $\{y: g(y) =1\}\subseteq \{x: f(x)=1\}$\footnote{In order to elucidate a correlation of the integral symbol with a measure, on terms of which it is defined, we take the liberty to return to use this integral symbol instead of $\displaystyle\int f\mathrm{d}x$.}. In fact, these two sets corresponding to their integrable functions $f$ and $g$ can possibly differ in a very restricted, 'ommittable' number of elements. By the last assumption, the following should hold:
\begin{eqnarray}
   \forall \epsilon > 0\Big( \mu\Big(\{x:f(x)=1\} - \{y: g(y)=1\}\Big)< \epsilon\Big)\\\iff \mathcal{M}\models^{H}\,"\int f\mathrm{d}\mu = \int g\mathrm{d}\mu". 
\end{eqnarray}
Since $\{x\in \vert M\vert:f(x)=1\} = \Vert \phi\Vert_{M}$ and $\{x\in \vert M\vert:g(x)=1\} = \Vert \psi\Vert_{M}$, for some $\phi$ and $\psi$ corresponding to $f(x)$ and $g(x)$ (resp.), we can reformulate lines (24)-(25) to the following ones:
\begin{eqnarray}
  \forall \epsilon> 0\Big(\mu\Big(\Vert \phi\Vert_{M} - \Vert \psi\Vert_{M}\Big)< \epsilon\Big)\\\iff \mathcal{M}\models^{H}\," \displaystyle\int\phi\mathrm{d}\mu = \displaystyle\int\psi\mathrm{d}\mu", 
\end{eqnarray}
where $\Vert \phi\Vert$ is an interpretation of $\phi$ in a given model $\mathcal{M}$. \\

 \textbf{\textit{II. H\'{a}jek's satisfiability for '=' in non-Dirac's case}.} In this subsection, H\'{a}jek's satisfiability for '=' in a more general, non-Dirac's sense will be introduced. In other words, we will refer to $\mathcal{L}(HLI)$-formulae $\phi$ with their truth values $\Vert \phi \Vert = \{x: f_{\phi}(x)\in A\}$ such that $A$ may belong to a broad class of subsets of $[0,1]$, not only -- as previously -- being a singleton. Namely, we will take into account the sets of the form $\{x\in\vert M\vert:f_{\phi}(x)> \alpha\}$ and their measures $\mu\{x\in\vert M\vert:f_{\phi}(x)> \alpha\}$\footnote{Let us note that it opens a realistic possibility to consider not only their set-theoretic complements in the form of the sets of the type  $\mu\{x\in\vert M\vert:f_{\phi}(x)\leq \alpha\}$ but also -- denumerable sums of the pairwise disjoint sets of this type.}. 
 
 For simplicity of further considerations, we will adopt the following modified version of Definition 4 of Lebesgue integral for (normalized) simple function -- based on the sets of the above type: 
 \begin{equation}
     \displaystyle\oint f\mathrm{d}\mu = \sum_{\rho_{\alpha}\in(0,1)} \rho_{\alpha}\cdot\mu\{x\in\vert M\vert: f_{\phi}(x)> \alpha\},
 \end{equation}
 where $\rho_{\alpha}\in(0,1)$ is a normalized coefficient defining $f$ as a simple function and $\{x\in\vert M\vert: f_{\phi}(x)> \alpha\}$ is a measurable set, $\alpha\in(0,1)$\footnote{We do not specify how $\rho_{\alpha}$ depends on $\alpha$. One of the reasonable way may be to define $\rho_{\alpha} = \inf_{x\in\Vert M\vert}\{f_{\phi}(x) > \alpha\}$. However, it does not concern us in this context.}. 
 
As previously, we are willing to require an arbitrarily small (smaller than arbitrary positive $\epsilon$) sum of measures of the difference set $A = \{x\in\vert M\vert: f_{\phi}(x)> \alpha_{i}\} - \{x\in\vert M\vert: g_{\psi}(x)> \alpha_{j}\}$, for some functions $f_{\phi}(x)$, $g_{\psi}(x)$\footnote{These functions are associated by $\Vert\cdots\Vert$-function to $\phi$ and $\psi$ respectively.} and fixed $\alpha_{i}, \alpha_{j}\in(0,1)$, for $i,j\in I$\footnote{$I$ forms at the most denumerable set of indices.}, and a (weak probabilistic) model $\mathcal{M}$. Finally, we also expect that the second set determined by $\alpha_{j}$ is contained in the first set determined by $\alpha_{i}$. It leads to the following definition of H\'{a}jek's satisfiability of the integral symbols for non-Dirac's case. 
\begin{ardef}(\textbf{H\'{a}jek's satisfiability for integral $\mathcal{L}(HLI)$-expressions with '='}.) Let us assume that $\mathcal{M}$ be a weak probabilistic model with a model universe $\vert M\vert$, $\alpha_{i}, \alpha_{j}\in(0,1)$, for each $i,j\in I$, where I is (at most) denumerable set of indices. Let also establish two $\alpha_{i}$ and $\alpha_{j}$ such that $\{x\in\vert M\vert: g_{\psi}(x)> \alpha_{j}\}\subseteq\{x\in\vert M\vert: f_{\phi}(x)> \alpha_{i}\}$ and let $f, g$ be non-simple Lebesgue integrable functions corresponding to the sets. We say that an $\mathcal{L}(HLI)$-formula "$\displaystyle\int f\mathrm{d}\mu = \displaystyle\int g\mathrm{d}\mu$" \, is\, \texttt{H\'{a}jek's satisfiable in $\mathcal{M}$}, symb. $\mathcal{M}\models^{H} "\displaystyle\int f\mathrm{d}\mu = \displaystyle\int g\mathrm{d}\mu$",   if and only if the following condition holds:
\begin{eqnarray*}
  \forall_{\epsilon > 0}\Big(\sum_{i,j\in I}\mu\Big(\{x\in\vert M\vert: f_{\phi(x)} > \alpha_{i}\} \\ - \{x\in\vert M\vert: g_{\psi(x)} > \alpha_{j}\}\Big)< \epsilon\Big). 
\end{eqnarray*}
   \end{ardef}
We can venture to generalize the condition for all $\mathcal{L}(HLI)$-quantifier expressions with '='. 
\begin{ardef}\textbf{H\'{a}jek's satisfiability for $\mathcal{L}(HJL)$-quantifiers} \\\textbf{in non-Dirac case}).
Let us assume that $\mathcal{M}$ be a weak probabilistic model with a model universe $\vert M\vert$, $\alpha_{i}, \alpha_{j}\in(0,1)$, for each $i,j\in I$, where I is (at most) denumerable set of indices. Let also establish two $\alpha_{i}$ and $\alpha_{j}$ such that $\{x\in\vert M\vert: g_{\psi}(x)> \alpha_{j}\}\subseteq\{x\in\vert M\vert: f_{\phi}(x)> \alpha_{i}\}$ and let $f, g$ be non-simple Lebesgue integrable functions corresponding to the sets. Let assume finally that $Q\phi(x), Q\psi(x)$ are some quantifier $\mathcal{L}$(HLI)-expressions, i.e. $Q\in\{\forall_{x}, \exists_{x}, \displaystyle\int()\mathrm{d}x\}$.

We say that an $\mathcal{L}(HLI)$-quantifier formula "$Q\phi(x) = Q\psi(x)$" \, is\, \texttt{H\'{a}jek's satisfiable in $\mathcal{M}$}, symb. $\mathcal{M}\models^{H} "Q\phi(x) = Q\psi(x)"$, if and only if the following condition holds:
\begin{eqnarray*}
  \forall_{\epsilon > 0}\Big(\sum_{i,j\in I}\mu\Big(\{x\in\vert M\vert: f_{\phi}(x) > \alpha_{i}\} \\- \{x\in\vert M\vert: g_{\psi}(x) > \alpha_{j}\}\Big)< \epsilon\Big).  
\end{eqnarray*}
(In terms of Definition 20, $e(Q\phi(x) = Q\psi(y))= 1$.)
\end{ardef}
\begin{exo} We show that the thesis of Fubini theorem  $(\mu 5)$. Therefore, let us assume that $\mathcal{M}$ is a weak probabilistic model. If $f: \vert \mathcal{M}\vert\times\vert \mathcal{M}\vert\to [0,1]$, then -- due to a simplified version of Definition 16 (see: the paragraph before it) -- we get 
\begin{equation*}
    \displaystyle\oint(\displaystyle\oint f(x,y)\mathrm{d}x)\mathrm{d}y = \sum_{a\in\vert\mathcal{M}\vert} (\sum_{b\in\vert\mathcal{M}\vert} f(a,b)\mu(b))\mu(a) =
\end{equation*}    
  \begin{equation*}
    \sum_{a,b\in\vert\mathcal{M}\vert}f(a,b)\mu(a)\mu(b) =\sum_{b\in\vert\mathcal{M}\vert} (\sum_{a\in\vert\mathcal{M}\vert} f(a,b)\mu(a))\mu(b) =\\
    \end{equation*}
    \begin{equation*}
       = \displaystyle\oint\displaystyle\oint f(x,y)\mathrm{d}y\mathrm{d}x.
    \end{equation*}
Assuming that the measure $\mu^{*}\{(x,y)\in\vert\mathcal{M}\vert^{2}: f(x,y) = 1\}$ = $\Vert f\Vert_{\mathcal{M}}$ is defined\footnote{It should be a product measure built up from measures concerning each variable $x$ and $y$.}, we can immediately infer from (26-27) that 
$\mathcal{M}\models^{H}  \displaystyle\oint(\displaystyle\oint f(x,y)\mathrm{d}x)\mathrm{d}y=\displaystyle\oint(\displaystyle\oint f(x,y)\mathrm{d}y)\mathrm{d}x$.

\end{exo}
Being equipped with all the required definitions, we can define HLI as an abstract logic. 
\begin{ardef} (\textbf{$HLI$ as an abstract logic}). 
HLI is a triple 
\begin{equation*}
    (\mathcal{K},\, Sent_{\mathcal{L}(HLI)}, \models^{H}),
\end{equation*}
 where 
 \begin{enumerate}
     \item $\mathcal{K}$ is a class of weak probabilistic models closed under isomorphism, renaming and reducts,
     \item $Sent_{\mathcal{L}}$ is a set of $\mathcal{S}$-sentences of $\mathcal{L}(HLI)$,
     \item $\models^{H}$ is H\'{a}jek's satisfaction relation given by Definition 20 and Definition 22.
 \end{enumerate}  
\end{ardef}

\section{The Lindstr\o m-type Theorem for Extensions of HLI and an Extended Outline of its Proof}
Having defined $HLI$ as an abstract logic in the model-theoretic treatment, we can venture to formulate and prove the Lindstr\o m's-type theorem for all extensions of $HLI$. We begin with introducing a terminological tissue of the proof argumentation leading to the main result. They cannot be presented in Section II as they are involved in a conceptual tissue elaborated in Section III.   

\subsection{A terminological tissue of Lindstr\o m's theorem and its proof.}

In order to perform the task, we need to consider both $HLI$ (as an abstract logic) and its extensions as equipped by unique binary relations to be called a \textit{H\'{a}jek approximation system} and denoted by $\lhd^{H}$. For each pair of formulae, say $\phi(x)$ and $\psi(x)$\footnote{We are especially interested in the pairs of the formulae as in the previous definitions of H\'{a}jek satisfiability.}  of a given abstract logic language $\mathcal{L}[L]$ -- this unique relation will intentionally encode some of the metalogical properties of the formulae. 

Some of the properties, such as Definition 23 of an abstract logic language $\mathcal{L}[L]$  or H\'{a}jek's satisfiability in their corresponding $\mathcal{L}$-structures in the sense of Definition 21 and Definition 22 have been presupposed or \textit{implicitly} assumed. 
It was previously made in some informal way. 
In this section, they will be explicitly indicated and formally expressed in terms of the following definition of $\lhd^{H}$-approximation system.

\begin{ardef}(\textbf{H\'{a}jek's} $\lhd^{H}$-\textbf{approximation system}).
Let $\mathcal{L}$ be an abstract logic. The binary relation $\lhd^{H}$ is said to be \texttt{H\'{a}jek's} $\lhd^{H}$-\texttt{approximation system} if and only if the following conditions hold:
\begin{enumerate}
    \item $\lhd^{H}$ is transitive,
    \item If $\phi\lhd^{H}\phi^{'}$ and $\phi\in \mathcal{L}[L]$, then $\phi^{'}\in\mathcal{L}[L]$.
    \item  If $\phi\lhd^{H} \phi^{'}$ and $\mathcal{M}\models ^{H}\phi$, then $\mathcal{M}\models^{H} \phi^{'}$.
\end{enumerate}
\end{ardef}
The following definition explains how to relate H\'{a}jek's satisfiability to the H\'{a}jek's $\lhd^{H}$-approximation, just introduced. These two notions meet together in the concept of \textit{approximate H\'{a}jek's satisfiability}.

\begin{ardef}
We say that $\mathcal{M}$ \texttt{approximately H\'{a}jek's satisfies} a formula $\phi\in\mathcal{L}[L]$, symb. $\mathcal{M}\models^{HA}\phi$, if for all $\lhd^{H}$-approximations $\phi^{'}$ of $\phi$, it holds $\mathcal{M}\models^{H} \phi^{'}$.
\end{ardef}
\begin{remark} Let us note that $\models^{HA}$ is weaker than $\models^{H}$, which is weaker than $\models$-satisfaction. 
\end{remark}
Because of the new conceptual scenario determined by H\'{a}jek's approximation satisfiability, we should redefine both Compactness Property and the Elementary Chain Condition for abstract logic with H\'{a}jek's approximation.  
We will say a theory (as a set of sentences) $\sigma$ of $\mathcal{L}$ is consistent if there
exists a structure $\mathcal{M}$ of $\mathcal{L}$ which approximately satisfies every sentence in $\Sigma$. In a similar way, we
will say that $\sigma$ is finitely consistent if every finite subset of $\sigma$ is consistent.

Let $(\mathcal{L}, \lhd^{H})$ be a logic with H\'{a}jek's approximations. For a structure with semantic integrals $\mathcal{M}$, let $\mathrm{Th}(\mathcal{L})^{HA}
(\mathcal{M})$ denote the theory (set of sentences) of $\mathcal{L}$ which are H\'{a}jek's approximately satisfied by $\mathcal{M}$. 
\begin{ardef}(\textbf{H\'{a}jek's Elementary substructure}.)
Let $\mathcal{M,N}$ be structures of $\mathcal{L}$\footnote{In all these contexts, we consider $\mathcal{L}$ by default as the pair $(\mathcal{L}, \lhd^{H})$.}. The structure $\mathcal{M}$ is said to be \texttt{H\'{a}jek's elementary substructure} of $\mathcal{N}$ and will be denoted by 
\begin{equation}
    \mathcal{M}\prec^{HA}_{\mathcal{L}} \mathcal{N}
\end{equation}
if and only if
$M\subseteq N$ and the structure $(\mathcal{N}, a)_{a\in\mathcal{M}}$ H\'{a}jek's approximately satisfies $\mathrm{Th}(\mathcal{L})^{HA}
((M, a)_{a\in \mathcal{M}})$\footnote{This defining is justified by the fact that the class of structures of a logic is assumed to be closed under expansions by constants. Recall that $\prec^{HA}$ is a relation of being elementary substructure (in terms of $\models^{HA}$ satisfiability).}.
\end{ardef}

\begin{property} (\textbf{Compactness}). Let $(\mathcal{L}, \lhd^{H})$ be an abstract logic with H\'{a}jek's approximation. $(\mathcal{L}, \lhd^{H})$ is said to be satisfied the \texttt{compactness theorem} if
it has the property that every theory of $\mathcal{L}$ which is finitely consistent is consistent.
\end{property}

\begin{property} (\textbf{The Elementary Chain Condition}.)
Let $(\mathcal{L}, \lhd^{H})$ be a logic with H\'{a}jek's approximations.  $(\mathcal{L}, \lhd^{H})$ satisfies the
\texttt{elementary chain condition} if the following holds. Whenever
\begin{equation*}
    \mathcal{M}_{0}\prec_{\mathcal{L}}^{HA} \mathcal{M}_{1}\prec_{\mathcal{L}}^{HA}\mathcal{M}_{2}\prec_{\mathcal{L}}^{HA}\ldots, (n< \omega)
\end{equation*}
there exists such a structure $\mathcal{M}$ of $\mathcal{L}$ -- uniquely determined by $\bigcup \mathcal{M}_{n}$ -- that $\mathcal{M}_{n}\prec_{\mathcal{L}}^{HA} \mathcal{M}$, for each $n< \omega$.
\end{property}
\begin{property} (\textbf{Weak Negation Property}).
Let $(\mathcal{L}, \lhd)$ be a logic with approximations. We say that $(\mathcal{L}, \lhd)$ has a weak negation property if and only if 
there exists a monadic operation $\neg^{weak}$ on $\mathcal{L}$-sentences such that
\begin{enumerate}
    \item  If $\phi \in \mathcal{L}[L]$, then 
$\neg^{weak}\phi \in \mathcal{L}[L]$;
    \item If $\phi \in \mathcal{L}[L]$ and $\mathcal{M}$ forms an $L$-structure of $\mathcal{L}$, then
    \begin{itemize}
        \item $\mathcal{M}\models^{H} \phi$ or $\mathcal{M}\models^{H}\neg^{weak}\phi$,
        \item for each $\lhd$-approximation $\phi^{'}$ of $\phi$ it holds\\
        $\mathcal{M}\models^{HA}\neg^{weak}\phi^{'}$ implies $\mathcal{M}\not\models^{HA}\phi$. 
    \end{itemize}
\end{enumerate}
\begin{exo}
Weak negation forms a standard negation for each logic systems -- considered as a logic with approximation $\lhd$, where $\lhd$ is a diagonal relation on sentences (i.e. each sentence is the only approximation of itself).
\end{exo}
\end{property}
\begin{ardef} (\textbf{Reducibility of an abstract logic $\mathcal{L}$ to $HLI$}.\footnote{The similar definition of reducibility may be introduced for any pair of abstract logics. Note that we do not require H\'{a}jek's satisfiability in $\mathcal{L}$.}) Let $(\mathcal{L}, \lhd^{H})$ be a logic with $\lhd$-approximation, and $(HLI, \lhd^{H})$ and $(\mathcal{L}, \lhd^{H})$  have the same structures. We say
that a sentence $\phi\in\mathcal{L}[L]$ is \textit{reducible} to $HLI$ if the following condition holds. For every  $\lhd$-approximation $\phi^{'}$ of $\phi$ there exist two sentences $\psi[\phi, \phi^{'}]$, $\psi^{'}[\phi, \phi^{'}]$ of $HLI$, such that
\begin{enumerate}
    \item $\psi[\rho, \rho^{'}]\lhd^{H}\psi^{'}[\phi, \phi^{'}]$ and
    \item If $\mathcal{M}$ is a structure for $\mathcal{L})$(and $HLI$), than
    \begin{itemize}
        \item $\mathcal{M}\models_{\mathcal{L}}^{H}\phi$ implies $\mathcal{M}\models^{H}\psi[\phi, \phi^{'}]$,
        \item $\mathcal{M}\models^{H}\psi^{'}[\phi, \phi^{'}]$ implies $\mathcal{M}\models^{H}_{\mathcal{L}}\phi^{'}$.
    \end{itemize}
\end{enumerate}
We say that $(\mathcal{L}, \lhd)$ forms an \textit{extension} of $(HLI, \lhd^{H})$ if every sentence of $\mathcal{L}$ is reducible to $HLI$. 
\end{ardef}
\subsection{The Proof of Lindstr\o m's type Theorem for HLI -- the Proof Idea and its Extended Outline}

\noindent\textbf{\textit{I. The proof idea}}. The proof of  Lindstr\o m's Theorem relies on showing that each abstract logic $(\mathcal{L}, \lhd)$ as an extension of $(\mathcal{L}, \lhd)$ is equivalent to the abstract logic $(HLI, \lhd^{H})$ (i.e., they have the same expressive power) if it satisfies both compactness theorem, the so-called elementary chain condition, and it is closed on negation. The proof itself is carried our by \textit{reductio ad absurdum} by showing that violating one of the conditions\footnote{It is assumed that $(\mathcal{L}, \lhd)$ is not reducible to $(HLI, \lhd^{H})$ in our proof.} generates a contradiction.

The argumentation line of the proof is based on two important but rather technical lemmas. Lemma 2 allows us to specify the initial situation for the construction of the required contradiction. More precisely, it ensures the existence of the following two $\mathcal{L}$-structures, $\mathcal{M}_{0}$ and $\mathcal{M}_{1}$ such that there exists $\theta\in\mathcal{L}$, which is satisfied (in H\'{a}jek's sense) in one of them but does not in the second one (more precisely: a weak negation of its approximation $\theta^{'}$ is satisfied here). This scenario will be disconfirmed by means of Lemma 2. This lemma -- together with the elementary chain condition -- warranties not only a linearly ordered sequence of structures $\mathcal{M}_{0}\prec \mathcal{M}_{1}\prec\ldots$, but also a structure $\mathcal{M}$ as the final element of that sequence and such that we get  $\mathcal{M}\models^{HA}\theta$ and $\mathcal{M}\models^{HA}\neg^{weak}\theta^{'}$, for the same $\theta\in\mathcal{L}$ and its $\lhd^{H}$-approximation $\theta^{'}$ as previously. The proof of Lemma 1 exploits Observation 1, which establishes some equivalence between H\'{a}jek's approximate satisfiability of a given theory $\Sigma\in HLI$ and H\'{a}jek's approximate satisfiability of a corresponding theory $\Sigma^{L}$ -- built up in a unique way from sentences of $\mathcal{L}$ dependent on elements of $\Sigma$. \\

\noindent\textbf{\textit{II. An extended outline of the proof}}. 
It is convenient to present the proof line of Lindstr\o m's Theorem beginning with formulating Observation 1. For that reason, for a theory $\Sigma\in (HLI, \lhd^{H})$, let 
\begin{equation*}
    \Sigma^{\mathcal{L}} = \{\psi[\rho^{'}, \rho^{''}] : \rho^{'}
\lhd^{H} \rho^{"}\,\,\mathrm{and}\,\, \rho \lhd^{H} \rho^{'},\,\, \mathrm{for\,\, some}\,\, \rho \in \Sigma\}.
\end{equation*}

\begin{obs} Let $\Sigma$ be a theory in $HLI$ and $\Sigma^{\mathcal{L}}$ defined as previously. Than it holds the following:
\begin{equation}
    \mathcal{M}\models^{HA}\Sigma\iff \mathcal{M}\models^{HA}_{\mathcal{L}}\Sigma^{\mathcal{L}}.
\end{equation}
\end{obs}
\begin{proof}
We show this property for the implication in one side.  Let $\mathcal{M}\models^{HA}\rho$, for some $\rho\in\Sigma$. It exactly means that
\begin{equation}
    \forall\rho^{'}\in\Sigma (\rho\lhd^{H}\rho^{'}\Rightarrow \mathcal{M}\models^{H} \rho). 
\end{equation}
Therefore, let $\rho, \rho^{'}\in\Sigma$ be such that $\rho\lhd^{H}\rho^{'}$. Thus, we can infer that $\mathcal{M}\models^{H}\rho$ from definition of H\'{a}jek's $\lhd^{H}$-approximation. Since $\mathcal{M}  
\models^{H}\rho\Rightarrow \mathcal{M}\models^{H} \psi[\rho, \rho^{'}]$, for some sentence $\psi[\rho, \rho^{'}]$ dependent on $\rho, \rho^{'}\in\Sigma$ such that $\rho\lhd^{H} \rho^{'}$, we can deduce that 
$\mathcal{M}\models^{H}\psi[\rho, \rho^{'}]$. Let now $\psi[\rho^{'}, \rho^{"}]$ be such a sentence that $\psi[\rho, \rho^{'}]\lhd^{H}\psi[\rho^{'}, \rho^{"}]$. Since $\mathcal{M}\models^{H}\psi[\rho, \rho^{'}]$, then also $\mathcal{M}\models^{H}\psi[\rho^{'}, \rho^{"}]$, for $\rho, \rho^{'},\rho^{"}$ as in definition of $\Sigma^{L}$. 

In order to show that also $\mathcal{M}\models^{HA}\psi[\rho^{'}, \rho^{''}]$, one needs to show that for any sentence $\phi^{'}[\rho^{'}, \rho^{''}]\in \Sigma^{\mathcal{L}}$, it holds 
\begin{equation*}
 \psi[\rho^{'}, \rho^{''}] \lhd^{H} \phi^{'}[\rho^{'}, \rho^{''}].   
\end{equation*}
Meanwhile, this condition is satisfied from the definition of $\lhd^{H}$-approximation. Thus, $\mathcal{M}\models^{HA}\psi[\rho^{'}, \rho^{"}]$.
\end{proof} 

\begin{remark} An alternative proof of Observation 1 might be carried out 
inductively -- due to the taxonomy of $\mathcal{L}(HLI)$ formulae and conditions of H\'{a}jek's satisfaction for them. For that reason, $\lhd^{H}$ should be established among $\mathcal{L}(HLI)$ sentences, in particular -- among the sentences of the form $\psi[Q\phi(x)=Q\phi(y), Q\chi(x) = Q\chi(y)]$, where $Q$ is a quantifier sign admissible in $\mathcal{L}(HLI)$\footnote{This proof would be less general and more extended and less beneficial from the perspective of further analysis, so we omit its details.}. 
\end{remark}

\begin{lemma}  Let $\mathcal{M}$ and $\mathcal{N}$ be structures with integrals such that $\mathcal{M}\prec^{HA} \mathcal{N}$. Then there
exists a structure with integrals $\mathcal{K}$ such that
\begin{enumerate}
    \item $\mathcal{M}\prec^{HA}_{\mathcal{L}} \mathcal{K}$,\,\,and  $\mathcal{N}\prec^{HA}_{\mathcal{L}}\mathcal{K}$.
\end{enumerate}
\end{lemma}
\begin{proof} In order to justify an existence of such a $\mathcal{K}$ structure, it is convenient to think about it as a 'cumulative' structure, as consisting of two sets of elements from $\vert \mathcal{M}\vert$ and $\vert\mathcal{N/M}\vert$. From a syntactic point of view, it is enough to show that the theory $\mathrm{Th}((\mathcal{M},a)_{a\in\mathcal{M}})\cup \mathrm{Th}(\mathcal{N},a,b)_{a\in\mathcal{M}, b\in\mathcal{N/M}})$ is consistent. Due to the Lemma 1 -- it enough to show that the theory 
$\mathrm{Th}((\mathcal{M},a)_{a\in\mathcal{M}})^{\mathcal{L}}\cup \mathrm{Th}((\mathcal{N},a,b)_{a\in\mathcal{M}, b\in\mathcal{N/M}})^{\mathcal{L}}$ is consistent. 
Since $\mathcal{L}$ has compactness property, it is enough to show that each finite subset of this theory is consistent in order to show that the whole theory is consistent.

For that reason, let us fix a finite set (of sentences) of
$\mathrm{Th}((\mathcal{N},a,b)_{a\in\mathcal{M}, b\in\mathcal{N/M}})^{\mathcal{L}}$ as the set 
\begin{equation*}
   \Sigma ^{\psi}= \{\psi[\rho_{1}^{'},\rho_{1}^{"}], \psi[\rho_{2}^{'},\rho_{2}^{"}],\ldots, \psi[\rho_{n}^{'}, \rho_{n}^{"}]\},
\end{equation*}
for such $\rho_{1}, \rho_{2},\ldots, \rho_{n}\in\mathrm{Th}((\mathcal{M},a)_{a\in\mathcal{M}})$ that $\rho_{i} < \rho_{i}^{'}< \rho_{i}^{"}$, for $i= 1,2\ldots, n$. 
We will find the finite structure extension on a base of $\mathcal{M}$ and $\mathcal{N/M}$, in which the set $\{\psi[\rho_{1}^{'},\rho_{1}^{"}], \psi[\rho_{2}^{'},\rho_{2}^{"}],\ldots, \psi[\rho_{n}^{'}, \rho_{n}^{"}]\}$ is satisfied (in the sense of $\models^{H}_{\mathcal{L}}$-satisfaction). It will exactly mean that this set in consistent.

Therefore, let $\hat{a_{1}},\hat{a_{2}},\ldots, \hat{a_{k}}$ and 
$\hat{b_{1}},\hat{b_{2}},\ldots, \hat{b_{l}}$ be (an exhaustive) list of names of elements from $\vert\mathcal{M}\vert$ and $\vert\mathcal{N/M}\vert$ (resp.) occurred in $\rho$'s. Let us observe that \begin{equation*}
    ((\mathcal{N}, a, b)_{a\in\mathcal{M}, b\in\mathcal{N/M}}) \models^{HA} \forall_i^{n}\, \rho_{i}(\hat{a_{1}},\hat{a_{2}},\ldots, \hat{a_{k}},\hat{b_{1}},\hat{b_{2}},\ldots, \hat{b_{l}}).
\end{equation*}
 Since $\rho_{i}^{'}$ is H\'{a}jek's $\lhd^{H}$-approximation of $\rho_{i}$, for each $i = 1,\ldots, n$ we can enlarge $(\mathcal{M}, a_{1},\ldots, a_{k})$ by such new elements $c_{1}, c_{2}\ldots c_{l}\in\vert\mathcal{M}\vert$ to obtain
\begin{equation}
 (\mathcal{M}, a_{1},\ldots, a_{k}, c_{1}, c_{2}\ldots c_{l}) \models^{H}\forall_{i}^{n}\rho^{'}_{i} (\hat{a_{1}},\hat{a_{2}},\ldots, \hat{a_{k}},\hat{c_{1}},\hat{c_{2}},\ldots, \hat{c_{l}}). 
\end{equation}
Since the paraphrazed Condition 2 of Property 3 (of weak negation). asserts that $\mathcal{M}\models \rho^{'}\Rightarrow \mathcal{M}\models \psi[\rho^{'}, \rho^{"}]$, for all $\rho^{'}, \rho^{"}$ such that $\rho^{'}< \rho^{"}$, we can deduce from (32)  the following:  
\begin{equation}
(\mathcal{M}, a_{1},\ldots, a_{k}, b_{1}, b_{2}\ldots b_{l}) \models_{\mathcal{L}}^{H} \psi_{i}[\rho^{'}, \rho^{"}],    
\end{equation}
for $i = 1,2\ldots, n$. It means that 
\begin{equation}
(\mathcal{M}, a_{1},\ldots, a_{k}, b_{1}, b_{2}\ldots b_{l}) \models_{\mathcal{L}} \Sigma^{\psi},    
\end{equation}
i.e. $\Sigma^{\psi}$ is consistent. Since $\Sigma^{\psi}$ was chosen arbitrarily, we can state that each finite subset of $\mathrm{Th}((\mathcal{M},a)_{a\in\mathcal{M}})^{\mathcal{L}}\cup \mathrm{Th}(\mathcal{N},a,b)_{a\in\mathcal{M}, b\in\mathcal{N/M}})^{\mathcal{L}}$ is consistent. Because of the compactness property for $\mathcal{L}$, we can assert that the whole theory is consistent as required.
\end{proof}

\begin{lemma} Let $(\mathcal{L}, \lhd^{H})$ be an abstract logic with H\'{a}jek's approximation. 
Suppose that $\theta$ is a sentence of $\mathcal{L}[L]$ that is not reducible to $HLI$. Then
there exist an $\lhd^{H}$-approximation $\eta^{'}$
of $\eta$ and analytic structures $\mathcal{M}$ and $\mathcal{K}$ such that
\begin{enumerate}
    \item $\mathcal{M}\prec^{HA}\mathcal{K}$\,\,($\mathcal{M}$ is H\'{a}jek's elementary substructure),
    \item $\mathcal{M}\models^{HA}\theta$, but $\mathcal{K}\models^{HA}\neg^{weak}\theta^{'}$.
\end{enumerate}
\end{lemma} 
\begin{proof} Because of the previous lemma -- it is enough to show that 
\begin{equation*}
    \Sigma = \mathrm{Th}^{HA}((\mathcal{M},a)_{a\in\mathcal{M}})^{\mathcal{L}} \cup \{\neg^{weak} \theta^{'}\}
\end{equation*}
is consistent, or that there
exists a structure, say $\mathcal{K }$, of $\mathcal{L}$ which approximately satisfies every sentence in (a theory) $\Sigma$ of $\mathcal{L}$. In particular, it will be $\mathcal{K}\models^{HA}\neg\theta^{'}$. Indeed, each finite subset of $\Sigma$ is H\'{a}jek's approximately satisfied by a finite extension of $\mathcal{M}$. Thus -- because of compactness property for $\mathcal{L}$ -- the whole $\Sigma$ is H\'{a}jek's approximately satisfied in a structure, which is as required in point 3). 
\end{proof}

\begin{theorem} (\textbf{Lindstr\o m's Theorem for extensions of HLI}.)
Let $(\mathcal{L}, \lhd)$ and $(HLI, \lhd^{H})$ be such that:
\begin{enumerate}
    \item $(\mathcal{L}, \lhd^{H})$ extends $(HLI, \lhd^{H})$, 
    \item  $(\mathcal{L}, \lhd^{H})$  has compactness property, elementary chain condition and it is closed on weak negation.
\end{enumerate}
Then $(\mathcal{L}, \lhd^{H})\equiv(HLI, \lhd^{H})$.
\end{theorem}
\begin{proof}
 Let us assume that $(\mathcal{L}, \lhd^{H})$ has compactness property, the elementary chain condition, is closed on a weak negation. Finally, let $(\mathcal{L}, \lhd^{H})$) extend $(HLI, \lhd^{H})$, but not conversely, i.e.,    there is such a formula $\theta$ in a language of $(\mathcal{L}, \lhd^{H})$ that is not reducible to $(HLI, \lhd^{H})$. This last condition together with Lemma 2 allows us to state that there exist such $\mathcal{L}$-structures $\mathcal{M}_{0}$ and $\mathcal{M}_{1}$ that $\mathcal{M}_{0}\prec \mathcal{M}_{1}$ and there exists such a $\lhd$-approximation $\theta^{'}$ of $\theta\in\mathcal{L}$ that:
\begin{equation}
    \mathcal{M}_{0}\models^{HA} \theta\,\,\,\mathrm{and}\,\,\, \mathcal{M}_{1}\models^{HA} \neg^{weak} \theta^{'}.
\end{equation}
Simultaneously, from the elementary chain condition and Lemma 1 used iteratively  -- we can infer that there exists for each sequence of $\mathcal{L}$-structures:
\begin{equation*}
    \mathcal{M}_{0}\prec_{\mathcal{L}}^{HA} \mathcal{M}_{1}\prec_{\mathcal{L}}^{HA}\mathcal{M}_{2}\prec_{\mathcal{L}}^{HA}\ldots
\end{equation*}
such a $\mathcal{L}$-structure $\mathcal{M}$ that
\begin{equation*}
    \mathcal{M}_{0}\prec_{\mathcal{L}}^{HA} \mathcal{M}_{1}\prec_{\mathcal{L}}^{HA}\mathcal{M}_{2}\prec_{\mathcal{L}}^{HA}\ldots \mathcal{M}.
\end{equation*}
In particular, taking now $\mathcal{M}_{0}\prec_{\mathcal{L}}^{HA} \mathcal{M}$ and  
$\mathcal{M}_{1}\prec_{\mathcal{L}}^{HA} \mathcal{M}$ we obtain:
$\mathcal{M}\models^{HA}_{\mathcal{L}}\theta$ and $\mathcal{M}\models^{HA}_{\mathcal{L}} \neg^{weak}\theta^{'}$. It generates a contradiction with point 2. of the weak negation property. Hence, $(\mathcal{L}, \lhd^{H})\equiv(HLI, \lhd^{H})$.  
\end{proof}

\section{Conclusions and closing remarks}

It has been shown how H\'{a}jek's Fuzzy Logic of Integrals -- described classically in \cite{hajek1998book} -- may be characterized in Lindstr\o m-style. Indeed, H\'{a}jek's system was considered as a minimal logic in a broader class of fuzzy logic systems suitable for describing continuous structures based on measures. Simultaneously, H\'{a}jek's logic constitutes the first formal system suitable to describe integrals and their properties which has its abstract logic-based depiction. Nevertheless, it required introducing a two-level satisfaction relation (H\'{a}jek's satisfiability and H\'{a}jek's approximate satisfiability), which increases the complexity of the considerations and potentially decreases the clarity of reasoning. Simultaneously, Lindstr\o m's-type characterizability for all continuous structures -- discussed in this context until today -- is executed through the compactness theorem and the elementary chain condition instead of the compactness theorem and Skolem-Loewenheim's theorem. 

An idea of Lindstr\o m's-type characterizability might also be extrapolated for other structures and entities of functional and real analysis, such as Banach L$^{2}$-spaces. Simultaneously, Lebesgue integrals might be exchanged by transforms or splots.
Unfortunately, a problem arises of a need for the appropriate formal logic systems to describe them and their typical properties. The system proposed in \cite{jobczykecai} constitutes rather a surrogate of the formal system, and it would require some further complements.

\section{Acknowledgement}

Mirna Dzamonja received funding from the European Union Horizon 2020 Research and Innovation Programme under the Maria Skłodowska-Curie grant agreement No. 1010232, FINTOINF. She also gratefully acknowledges her association with IHSPT at Université Panthéon-Sorbonne, the School of Mathematics, The University of East Anglia, and the Academy of Sciences and Arts in Bosnia and Herzegovina, ANUBiH. 

Krystian Jobczyk is grateful to the Polish Fulbright Commission for the Senior Award Grant supporting his research stay in the Saul Kripke Center at the City University of New York, where the paper was written.




 \begin{description}
     \item Mirna  D\v{Z}AMONJA, Institut de Recherche en Informatique Fondamentale CNRS et l’Universit\'{e} de Paris, 8 Place Aurelie Nemours, Paris Cedex 13, 75205, France
     \item Krystian JOBCZYK, Department of Applied Computer Science, AGH University of Science and Technology, al. Mickiewicza 30, 30 -059, Krak\'ow, Poland.
 \end{description}

\end{document}